\documentclass[12pt]{amsart}
\usepackage[a4paper,margin=0.9in]{geometry}
\usepackage{verbatim}
\usepackage[colorlinks,citecolor=magenta,linkcolor=black]{hyperref}
\pdfpagewidth=\paperwidth \pdfpageheight=\paperheight
\usepackage{amsfonts,amssymb,amsthm,amsmath,eucal,tabu,url}
\usepackage{pgf}
 \usepackage{array}
 \usepackage{tikz-cd}
 \usepackage{pstricks}
 \usepackage{pstricks-add}
 \usepackage{pgf,tikz}
 \usetikzlibrary{automata}
 \usetikzlibrary{arrows}
 \usepackage{indentfirst}
\usepackage{tabularx} 
\usepackage{graphicx} 
\usepackage[all]{xy}
\usepackage{color}
\usepackage{enumerate}

\theoremstyle{plain}
\newtheorem{thm}{Theorem}[section]
\newtheorem{theorem}[thm]{Theorem}

\newtheorem{proposition}[thm]{Proposition}
\newtheorem{corollary}[thm]{Corollary}
\newtheorem{conjecture}[thm]{Conjecture}

\theoremstyle{definition}

\newtheorem{remark}[thm]{Remark}
\newtheorem{example}[thm]{Example}

\newtheorem{question}[thm]{Question}

\newtheorem{assumption}[thm]{Assumption}

\newtheorem{thevarthm}[thm]{\varthmname}

\newenvironment{varthm*}[1]{\trivlist\item[]{\bf #1.}\it}{\endtrivlist}
\newtheorem{custom}{{\bf{Main Theorem}}}

\usepackage[colorinlistoftodos]{todonotes}
\newcommand\pp[1]{\todo[color=cyan!50]{#1}} 

\renewcommand\ge{\geqslant}
\renewcommand\geq{\geqslant}
\renewcommand\le{\leqslant}
\renewcommand\leq{\leqslant}

\newcommand\CL{\mathcal{CL}}
\newcommand\oCL{\overline{\mathcal{CL}}}

\newcommand\be{\begin{eqnarray*}}
\newcommand\ee{\end{eqnarray*}}

\newcommand\bZ{\mathbb Z}

\renewcommand\pp{\mathbb P}

\newcommand\newop[2]{\def#1{\mathop{\rm #2}\nolimits}}
\newop\log{log}
\newop\ord{ord}
\newop\Gal{Gal}
\newop\SL{SL}
\newop\Bl{Bl}
\newop\mult{mult}
\newop\mass{mass}
\newop\div{div}
\newop\codim{codim}
\newop\sing{sing}
\newop\vdim{vdim}
\newop\edim{edim}
\newop\Ass{Ass}
\newop\size{size}
\newop\reg{reg}
\newop\satdeg{satdeg}
\newop\supp{supp}
\newop\Neg{Neg}
\newop\Nef{Nef}
\newop\Nefh{Nef_H}
\newop\Eff{Eff}
\newop\Zar{Zar}
\newop\MB{MB}
\newop\MBxC{MB\mathit{(x,C)}}
\newop\NnB{NnB}
\newop\Bigg{Big}
\newop\Effbar{\overline{\Eff}}

\newcommand{\Z}{\mathbb Z}

\newcommand{\C}{\mathbb C}

\def\keywordname{{\bfseries Keywords}}%
\def\keywords#1{\par\addvspace\medskipamount{\rightskip=0pt plus1cm
\def\and{\ifhmode\unskip\nobreak\fi\ $\cdot$
}\noindent\keywordname\enspace\ignorespaces#1\par}}
\def\subclassname{{\bfseries Mathematics Subject Classification
(2020)}\enspace}
\def\subclass#1{\par\addvspace\medskipamount{\rightskip=0pt plus1cm
\def\and{\ifhmode\unskip\nobreak\fi\ $\cdot$
}\noindent\subclassname\ignorespaces#1\par}}

\begin{document}
\title{Bounds for characteristic numbers of  conic--line arrangements in the plane}

\author{Rita Pardini and Piotr Pokora}
\date{\today}
\begin{abstract}
The main aim of the note is to provide an upper--bound for the  characteristic number of conic--line arrangements with ordinary singularities in the complex projective plane.
\subclass{14C20, 14N25, 14J29}
\end{abstract}
\maketitle
\setcounter{tocdepth}{1}
 
\tableofcontents

\section{Introduction}
The main aim of the present note is to answer a question regarding log--surfaces that are defined as the complement in $\pp^2_{\mathbb C}$ of a conic--line arrangements with ordinary singularities. This is a very classical problem in the theory of algebraic surfaces and it dates back (at least) to work due to Iitaka \cite{Iitaka}. 
An arrangement of curves $\mathcal{D} = \{D_{1}, \dots, D_{m}\} \subset \mathbb{P}^{2}_{\mathbb{C}}$ is a collection of smooth projective curves with $m\geq 3$; 
$\mathcal{D}$ is said to be simple crossing if any two of the $D_i$  intersect transversally. For $2\leq k \leq m-1$, a $k$--point is a point in $\mathcal{D}$ which belongs to exactly $k$ curves. Following \cite{BHH87}, the number of $k$--points will be denoted by $t_{k}$.

Let $\mathcal{D} = \{D_{1}, \ldots, D_{m}\} \subset \mathbb{P}^{2}_{\mathbb{C}}$ be a simple crossing arrangement. Consider the blow--up $\tau \colon Y \rightarrow \mathbb{P}^{2}_{\mathbb{C}}$ of the complex plane at all the $k$--points of $\mathcal{D}$ with $k\geq 3$, and denote by $\overline{\mathcal{D}}$ the reduced total transform of $\mathcal{D}$ under $\tau$, which is a simple normal crossings divisor. The logarithmic Chern numbers $\bar{c}_{1}^{2}(Y,\overline{\mathcal{D}})$, $\bar{c}_{2}(Y,\overline{\mathcal{D}})$ of the  pair $(Y,\overline{\mathcal{D}})$ are defined as the Chern numbers of the vector bundle $\Omega^1_Y(\log \overline{\mathcal{D}})$ and 
they can be computed in terms of the numerical invariants of $\mathcal{D}$, as follows (see \cite[p. 27]{Urzua}):
\begin{gather}\label{eq: chern-general}
\bar{c}_{1}^{2}(Y,\overline{\mathcal{D}}) = 9 - \sum_{i=1}^{m}D_{i}^{2} + 4\sum_{i=1}^{m}(g(D_{i})-1) + \sum_{r\geq 2}(3r-4)t_{r},\\
\bar{c}_{2}(Y,\overline{\mathcal{D}}) = 3 + 2\sum_{i=1}^{m}(g(D_{i})-1) + \sum_{r\geq 2}(r-1)t_{r}.\nonumber 
\end{gather}
Following Hirzebruch's paper \cite{Hirzebruch}, we define the characteristic number of $(Y,\overline{\mathcal{D}})$ (aka the log-Chern slope of the log surface $(Y,\overline{\mathcal{D}})$) as
$$\gamma(\mathcal{D}) := \frac{\bar{c}_{1}^{2}(Y,\overline{\mathcal{D}})}{\bar{c}_{2}(Y,\overline{\mathcal{D}})}.$$
It is well--known by \cite{Miyaoka77,Wahl} that $\gamma(\mathcal{D}) \leq 3$, but it is natural to wonder if we can get a sharper upper bound on the slope when considering a particular class of plane curves.

If we restrict our attention to line arrangements $\mathcal{L} \subset \mathbb{P}^{2}_{\mathbb{C}}$, then by a result due to Sommese \cite[(5.3) Theorem]{Som} one has
$$\gamma(\mathcal{L}) \leq \frac{8}{3},$$
and we get  equality if and only if $\mathcal{L}$ is the dual Hesse arrangement consisting of $m=9$ lines and with $t_{3}=12$. If we now look at the arrangements of smooth conics $\mathcal{C} \subset \mathbb{P}^{2}_{\mathbb{C}}$ with only ordinary singularities, then the second author showed in \cite{Pokora1} that
$$\gamma(\mathcal{C}) < \frac{8}{3},$$
but we do not know much about conic arrangements with characteristic numbers close to $8/3$. In his Ph.D. thesis \cite[Section 2.5.1]{Urzua}, G. Urz\'ua asked a question that can be formulated as follows.
\begin{question}
Let $\mathcal{CL}\subset \mathbb{P}^{2}_{\mathbb{C}}$ be an arrangement consisting of $d$ lines and $k$ smooth conics such that $t_{d+k}=0$ and all the singularities of $\mathcal{CL}$ are \textbf{ordinary}. Is it true that
$$\gamma(\mathcal{CL}) \leq \frac{8}{3}$$
always holds?
\end{question}
The main result of this note provides a positive,  sharper, answer to the above question under the assumption that $k\ge 3$.

\begin{custom}[see Theorem \ref{charn}]
Let $\mathcal{CL} \subset \mathbb{P}^{2}_{\mathbb{C}}$ be an arrangement of $d$ lines and $k\geq 3$ smooth conics such that $t_{d+k}=0$. Then $\gamma(\mathcal{CL}) < \frac{8}{3}$.
\end{custom}

Here is the structure of our paper. In Section \ref{sec: abel-prel}, we show how to construct, for any prime $p$, a $\mathbb Z_p^{k+d-1}$--cover of $\pp^2_{\mathbb{C}}$ branched precisely on the curves of the given conic--line arrangement. In Section \ref{sec: abcovers}, we prove Hirzebruch-type inequalities for conic--line arrangements with ordinary singularities -- this result allows us to prove in Section \ref{sec:bounds} our Theorem \ref{charn}, which is the main result of our note.

We work over the complex numbers and all varieties are projective.

\section{Preliminaries on abelian covers} \label{sec: abel-prel} 
For the reader's convenience  here we  recall  from \cite[\S~2]{RP} the structure theorem for abelian covers in the special case when the Galois group is $G\cong \mathbb{Z}_{p}^{m}$ for a prime $ p$. 

Consider a $G$-cover 
$f \colon  X \rightarrow Y$ with $Y$ smooth and $X$ a normal variety. The irreducible components of the branch divisor of $f$ are labeled  by pairs of the form $(H, \Psi)$, where $H < G$ is a cyclic subgroup of $G$ and $\Psi$ is a generator of $H^*:={\rm Hom}(H,\mathbb C^*)$,  the group of characters of $H$.
 If we fix  a primitive $p$--th root $\xi$ of $1$, we can use instead the non-zero elements of $G$  to label the components of the branch divisor of $f$, thanks to  the following bijection:
$$ \quad g \in G \setminus \{0\} \quad \iff \quad (H,\Psi) \text{ is a pair as above}$$ 
that sends a non-zero element $g \in G$ to the pair consisting of the subgroup $H$ generated by $g$ and of the character $\Psi$ of $H^*$ such that $\Psi(g)=\xi$.
Similarly, we have the identification:
$$G^{*} = {\rm Hom}(G,\mathbb{C}^{*}) \cong {\rm Hom}(G, \mathbb{Z}_{p}),$$
defined by sending $\psi \in {\rm Hom}(G, \mathbb{Z}_{p})$ to the character 
$$g \mapsto \xi^{\psi(g)}.$$

Finally, given a class $\alpha \in \mathbb{Z}_{p}$, we write $\hat{\alpha}$ for the only representative of $\alpha$ such that $0 \leq \alpha < p$. Then, given a $G$--cover, we have the following
\begin{equation}\label{eq: fundrel}
pL_{\chi} \equiv_{{\rm lin}} \sum_{g\neq 0} \widehat{\chi(g)}D_{g}, \quad \forall \chi \in G^{*},
\end{equation}
where $\equiv_{{\rm lin}}$ denotes linear equivalence of divisors.
Conversely, if $\chi_{1}, \dots, \chi_{m} \in G^{*}$ is a basis and $L_{\chi_1}, \dots, L_{\chi_{m}}$ are line bundles satisfying 
\begin{equation*}
pL_{\chi_{i}} \equiv_{{\rm lin}} \sum_{g\neq 0} \widehat{\chi_{i}(g)}D_{g}.
\end{equation*}
then there exists a $G$--cover $f \colon X \rightarrow Y$ branched along the $D_{g}$'s.  

Let $\mathcal{CL} = \{\ell_{1}, \ldots, \ell_{d}, C_{1}, \ldots, C_{k}\} \subset \mathbb{P}^{2}_{\mathbb{C}}$ be an arrangement consisting of $d$ lines and $k$ smooth conics such that $t_{d+k}=0$. 

Let $p$ be an odd prime. We are going to show the existence of an abelian covering with $G \cong (\mathbb{Z}_{p})^{d+k-1}$ branched along $\mathcal{CL}$. This is an instance of a more general construction described in \cite[\S~2.2]{AP}.
Let $G$ be the group defined by the following exact sequence:
\begin{equation}
\label{seq}
    0 \longrightarrow \mathbb{Z}_{p} \xrightarrow{\quad \delta \quad} \mathbb{Z}_{p}^{d+k} \longrightarrow G \longrightarrow 0 
\end{equation}
with the mapping $\delta$ given by
$$1 \rightarrow (\underbrace{1, \dots, 1}_{k \text{ times}}, \underbrace{2, \dots, 2}_{d \text{ times}}).$$
Denote by $g_{i} \in G$ the image of the $i$--th vector of the canonical basis of $\mathbb{Z}_{p}^{k+d}$. We set
$$D_{g_{i}} = \ell_{i} \,\, \text{ for } \,\, i \in \{1, \dots, d\},$$
$$D_{g_{i}} = C_{i-d} \,\, \text{ for } \,\, i \in \{d+1, \dots, k+d\}$$
and $D_g=0$ for all the remaining $g\in G$.
Dualizing the sequence \eqref{seq}, we obtain:
$$0\longrightarrow G^{*} \longrightarrow \left(\mathbb{Z}_{p}^{d+k}\right)^* \rightarrow \mathbb{Z}_{p} \longrightarrow 0,$$
i.e., the character $\chi = (x_{1}, \dots, x_{k},y_{1}, \dots, y_{d}) \in \left(\mathbb{Z}_{p}^{d+k}\right)^* $ is in $G^{*}$ if and only if $$\sum_{i=1}^{k}x_{i} + 2\sum_{i=1}^{d}y_{i} = 0 \text{ holds in } \mathbb{Z}_{p}.$$
Equation \eqref{eq: fundrel} now reads
$$pL_{\chi} \equiv_{{\rm lin}}\sum_{i=1}^{k} \hat{x_{i}}\ell_{i} + \sum_{i=1}^{d} \hat{y_{i}}C_{i} \equiv_{{\rm lin}} \bigg( \sum_{i}^{k}\hat{x_{i}} + 2\sum_{i=1}^{d}\hat{y_{i}}\bigg)\mathcal{O}_{\mathbb{P}^{2}_{\mathbb{C}}}(1)$$
and it can be solved for any $\chi \in G^{*}$. This shows that there exists a $G$--cover with $G = \mathbb{Z}_p^{d+k-1}$ branched along $\mathcal{CL}$. 

The same construction can be performed also for $p=2$, if $d>1$: the $\bZ_2^{k+d-1}$--cover that one obtains is the fiber product of the Kummer cover with Galois group $\bZ_2^{d-1}$ branched on the lines and of the $k$ double covers branched on $C_{1},\ldots ,C_{k}$, respectively.

\begin{remark}
One can construct $\Z_p^{d+k-1}$--covers of $\pp^2_{\C}$ branched over the components of $\CL$ for any odd integer $p$, not necessarily prime. Here we only treat the case when $p$ is prime because it is sufficient for our purposes and is easier to describe. 
\end{remark}
In the next section we will use these constructions to obtain  Hirzebruch-type inequalities for conic--line arrangements in the complex plane with only ordinary singularities.

\section{Inequalities for  conic-line arrangements via abelian covers}\label{sec: abcovers}

In this section we derive some inequalities for the numerical invariants of conic-line arrangements. The technique, that goes back to Hirzebruch's work on line arrangements \cite{Hirzebruch83}, boils down to constructing singular abelian covers branched on the curves of the arrangement, explicitly computing the desingularization and then applying  to it the Bogomolov--Miyaoka--Yau inequality or the log version due to Miyaoka  \cite{M84}.
Proposition \ref{HirCL} is the key ingredient for the proof of our main result in the next section.

Throughout all the section we make the following
\begin{assumption}
 $\mathcal{CL} = \{\ell_{1}, \ldots, \ell_{d}, C_{1}, \ldots, C_{k}\} \subset \mathbb{P}^{2}_{\mathbb{C}}$ be an arrangement of lines and smooth conics admitting only ordinary singularities and such that $k\geq 3$ and $t_{d+k}=0$. 
\end{assumption}

Let $p$ be a prime. As explained in \S \ref{sec: abel-prel}, there is an abelian cover $f_p \colon  X_p \rightarrow \mathbb{P}^{2}_{\mathbb{C}}$ with Galois group $G\cong \mathbb{Z}_p^{d+k-1}$ branched  on the curves of $\mathcal{CL}$ (if $p=2$ we also assume that $d\ge 3$ or $d=0$, and in the latter case the Galois group is $\Z_2^k$). 

 We use the notation introduced in \S \ref{sec: abel-prel}. Assume first that $p$ is odd. By \cite[Proposition~3.1]{RP} a point $Q \in X$ is singular if and only if $f(Q)$ is a point of $\mathcal{CL}$ of multiplicity $\geq 3$, since any subset of the $g_i$ of cardinality $\le d+k-1$ is independent in $G\cong \mathbb{Z}_p^{d+k-1}$ (notation as in \S \ref{sec: abcovers}). Denote by $\tau \colon Y \rightarrow \mathbb{P}^{2}_{\mathbb{C}}$ the blow--up of the complex projective plane at all singular points of $\mathcal{CL}$ of multiplicity $\geq 3$ and consider the following diagram obtained by taking base change and normalization:
\begin{equation}\label{eq: diag1}
\xymatrix{
W_p\ar[d]^{\sigma_p} \ar[r]^{\rho_p}&X_p\ar[d]^{f_p}\\
Y \ar[r]^{\tau}         & \pp^2  .}
\end{equation}
Now we are going to show that $W_{p}$ is smooth. Let $P\in \pp^2$ be a point of multiplicity $r \geq 3$ of $\mathcal{CL}$ and let $g_{i_1},\dots g_{i_r}\in G$ be the elements corresponding to the curves of the arrangement containing $P$. By the normalization algorithm of \cite[\S~3]{RP}, the exceptional curve $E_P$ of $\tau$ lying over $P$ appears in the branch locus of $\sigma$ with label $g_P:= g_{i_1}+ \dots + g_{i_r}$, unless $g_p=0$. Since $r<d+k$ by the assumption, as already noted above we have $g_P\ne 0$ and so $E_P$ is contained in the branch locus of $\sigma_p$. By the criterion recalled above, $W$ is smooth over $E_P$ since $g_P$ and $g_{i_j}$ are independent for every $j=1, \dots, r$.  Summing up, $W_p$ is a smooth cover of $Y$ branched over $\oCL:=\sum_PE_P+\sum_{i=1}^d\bar \ell_i+\sum_{j=1}^k\bar C_j$, where $E_P$ are the $\tau$--exceptional curves and $\bar \ell_i$, $\bar C_j$, are the strict transforms of $\ell_i$, $C_j$.

Assume now $p=2$; in this case we also assume that either $d=0$ or $d\ge 3$ and there is no point lying on all the lines and on no conic of $\CL$. Under these assumptions all the above claims are still true, except for the fact that if $d=0$ the Galois group is $\Z_2^k$. 

The canonical class of $W_p$ is the pull-back of the $\mathbb Q$--divisor:
\begin{multline}\label{eq: Kp}
K_p=K_Y +\frac{p-1}{p} \bigg( \sum_PE_P+\sum_{i=1}^d\bar \ell_i+\sum_{j=1}^k\bar C_j \bigg) =\\ \tau^*K_{\pp^2}+\frac{2p-1}{p}\sum_PE_P+\frac{p-1}{p}\bigg(\sum_{i=1}^d\bar \ell_i+\sum_{j=1}^k\bar C_j \bigg).
\end{multline}
If $p<q$, then $K_p<K_q$; we  set  $$K:=\lim_pK_p=K_Y+\sum_PE_P+\sum_{i=1}^d\bar \ell_i+\sum_{j=1}^k\bar C_j.$$ Note that $K$  is  the logarithmic canonical divisor $\bar {c_1}(Y, \overline{\mathcal{CL}})$.

We start by studying positivity properties of $K_p$ or, equivalently, of $K_{W_p}$. Since $\mathcal{CL}$ contains conics,  any line $\ell_i$ of $\mathcal{CL}$ meets the rest of the arrangement in at least two points. We say that a line of $\mathcal{CL}$ is {\em exceptional} if it meets the rest of the arrangement at exactly two points.
\begin{example}\label{ex: exceptional}
 Let $P_1,P_2, P_3$ be non--collinear points in $\pp^2$,  let $\ell_i$ be the line joining $P_j$ and $P_k$, where $i,j,k$ is a permutation of 1,2,3, and let $\mathcal{CL}$ be the arrangement $\{\ell_1,\ell_2,\ell_3, C_1,\ldots,C_k\}$ where the $C_j$ are conics containing $P_1,P_2,P_3$. The lines $\ell_1,\ell_2,\ell_3$ are exceptional for $\CL$. Take the Cremona transformation centered at $P_1, P_2,P_3$ and let $\mathcal L$ be the  arrangement consisting of the strict transforms of $C_1,\dots C_k$ and of the exceptional lines corresponding to $P_1,P_2, P_3$. 
Then $\mathcal L$ is a line arrangement  of degree $k+3$ that contains 3 lines that meet the rest of the arrangement only at double points of $\mathcal L$. Note that the Cremona transformation induces an isomorphism of the complements of $\CL$ and $\mathcal L$.

 It is not hard to see that a conic--line arrangement has at most three exceptional lines and that if it has three exceptional lines then it is the arrangement $\CL$ that we have just described.
 \end{example}

\begin{proposition}\label{prop: positivity} Let $p$ be a prime. In the above setting: 
\begin{enumerate} 
\item $K_p$ is effective;
\item $K_p$ is big for $p\ge 3$;
\item $K_p$ is nef iff $\CL$ contains no exceptional line; 
\item $K$ is nef and big.
\end{enumerate}
\end{proposition}
\begin{proof}
(i) Since $K_p<K_q$ if $p<q$, it is enough to prove the claim for $p=2$. Let  $P_1,\dots, P_m$  be the points of multiplicity $\ge 3$ of $\CL$ and let $\mu_t\le 3$ be the multiplicity of $P_t$ for $C_1+C_2+C_3$, and let $E_t$ be the exceptional curve corresponding to $P_t$, $t=1,\dots, m$.  Then \eqref{eq: Kp} can be rewritten for $p=2$ as 
\begin{gather*}
K_2= -\frac12\tau^*(C_1+C_2+C_3)+\frac 32\sum_{t=1}^m E_t+\frac{1}{2}\bigg(\sum_{i=1}^d\bar \ell_i+\sum_{j=1}^k\bar C_j \bigg)=\\
\sum_{t=1}^m \frac{3-\mu_t}{2} E_t+\frac{1}{2}\bigg(\sum_{i=1}^d\bar \ell_i+\sum_{j=4}^k\bar C_j \bigg) \ge 0
\end{gather*}
\medskip

(ii) Again, it is enough to prove the claim for $p=3$. Arguing as above we have 
\begin{gather*}
K_3=\frac 16 (\bar C_1+\bar C_2+\bar C_3)+\sum_{t=1}^m\frac{10-3\mu_t}{6} E_t+\frac{2}{3}\bigg(\sum_{i=1}^d\bar \ell_i+\sum_{j=4}^k\bar C_j \bigg)\ge \\
\frac {1}{18} \bigg(\sum_{t=1}^s 3 E_t + \bar C_1+\bar C_2+\bar C_3 \bigg)\ge\frac 1 3 \tau^*h,
\end{gather*}
where $h$ is the class of a line in $\pp^2$. Since $\tau^*h$ is big, $K_3$ is also big.
\medskip

(iii) The divisor  $K_{W_p}=\sigma_p^*K_p$ is effective by (i),  and therefore  it is nef if and only if $W_p$ contains no $(-1)$--curve. A $(-1)$--curve of $W_p$ maps to a curve $\Gamma$ of $Y$ with $K_p\Gamma<0$. Since by (i) $K_p$ is effective supported on the components of $\oCL=\sum_PE_P+\sum_{i=1}^d\bar \ell_i+\sum_{j=1}^k\bar C_j$ our $\Gamma$ must be a component of $\CL$.  If $\Gamma$ is the strict transform of an exceptional line $\ell_i$ of $\CL$, then $\Gamma^2=-1$. In fact, all the conics of $\CL$ must contain the points $P_1, P_2$ of intersection of $\ell_i$ with the rest of the arrangement, so $P_1$ and $P_2$ have multiplicity $\ge 4$ for $\CL$. 
By the Hurwitz formula the components of $\sigma_p^*\Gamma$ are rational. Since the stabilizer of $\sigma_p^*\Gamma$ has order $p^2$ by construction and $\Gamma$ is in the branch locus we see that $\sigma_p^*\Gamma$ is a disjoint union of $(-1)$--curves.
The only other case in which a component $\Gamma$ of $\oCL$ pulls back to a union of rational curves of $W_p$ is when $p=2$ and $\Gamma$ meets the rest of the branch locus at three points, namely $\Gamma$ is either the strict transform of a line blown up at three points or it is the exceptional curve over a point of multiplicity $3$ of $\CL$. In the either case, the stabilizer of $\sigma_2^*\Gamma$ has order $8$, therefore the components of $\sigma_2^*\Gamma$ have even self-intersection, equal to $2\Gamma^2$, and are not $(-1)$--curves.

(iv) The divisor $K$ is big by (ii). Since it is effective  supported on the components of $\oCL$, it is enough to check $K\Gamma\ge 0$ for every component $\Gamma$ of $\oCL$. Since $K\Gamma=\lim_p K_p\Gamma$, by the proof of (iii) it is enough to show $K\Gamma\ge 0$ when $\Gamma$ is the strict transform of an exceptional line of $\CL$. In that case a direct computation shows that $K\Gamma=0$.
\end{proof}

The following is an immediate consequence of Proposition \ref{prop: positivity}:

\begin{corollary}\label{cor: positivity} In the above set-up the following hold:
\begin{enumerate} 
\item $W_p$ has non-negative Kodaira dimension;
\item $W_p$ is of general type  for $p\ge 3$;
\item $W_p$ is minimal iff $\CL$ contains no exceptional line; 
\item $(Y,\oCL)$ is of log general type. 
\end{enumerate}
\end{corollary}

For an arrangement of curves $\mathcal C \subset \mathbb{P}^{2}_{\mathbb{C}}$ with only ordinary singularities,  we denote by $t_r$ the the number of points of multiplicity $r$ and we set 
$$ f_{0} = \sum_{r\geq 2}t_{r},\quad f_{1} = \sum_{r\geq 2}rt_{r}, \quad f_{2} = \sum_{r\geq 2}r^2t_{r}.$$
With the above  notation we have:
\begin{gather}\label{eq: chern-numbers}
\bar c_1(Y, \oCL)^2=9 -8k-5d+3f_1-4f_0,\\
\bar c_2(Y, \oCL)=3-2k-2d+f_1-f_0.\nonumber
\end{gather}

Now we are ready to present our Hirzebruch-type inequalities for conic--line arrangements. For $d=0$ next result recovers \cite[Theorem 3.1]{Tang}.

\begin{proposition}\label{HirCL}
Let $\mathcal{CL} = \{\ell_{1}, \dots, \ell_{d}, C_{1}, \dots, C_{k}\} \subset \mathbb{P}^{2}_{\mathbb{C}}$ be an arrangement of $d$ lines and $k\geq 3$ smooth conics admitting only ordinary singularities and such that $t_{d+k}=0$. Then one has
$$5k + t_{2} + t_{3} \geq d + \sum_{r\geq 5}(r-4)t_{r}.$$
\end{proposition}
\begin{proof}
We will use the notation introduced so far. We compute the Chern numbers of  the cover $\sigma_3\colon W_3\to Y$.  Since this is a rather standard computation, see \cite{BHH87, PSz2023}, we only outline it. 
We start with the Euler characteristic of $W$, namely 
$$e(W_{3})/3^{d+k-3}=9\cdot(3-2k-2d+f_{1}-f_{0})+6\cdot (d+k-f_{1}+f_{0}) + f_{1}-t_{2} = $$
$$27-12k-12d+4f_{1}-3f_{0}-t_{2}.$$

Next, we compute $K_{W_3}^{2}$. We have seen above that $K_{W_3}=\sigma_3^*K_3$, where
\begin{gather*}
K_3=K_Y+\frac{2}{3}(\sum_PE_P+\sum_{i=1}^d\bar \ell_i+\sum_{j=1}^k\bar C_j)\\
\end{gather*}
where the summation above is taken over all singular points $P$ of $\mathcal{CL}$ having multiplicity $\geq 3$. 
Following the path of \cite[Kapitel 1.3]{BHH87}, we get
$$9(K_3)^2= 9 (9-8k-5d + 3f_{1}-4f_{0})+ 12 (k+d-f_{1}+f_{0}) + 4k+d+f_{1}-f_{0}+t_{2}=$$
$$81-56k-32d+16f_{1}-25f_{0}+t_{2}.$$
By Corollary \ref{cor: positivity}, the surface $W_3$ is of general type,  therefore we can apply the Bogomolov--Miyaoka--Yau inequality, hence we have
$$\frac{3e(W_3)-K_{W_3}^{2}}{3^{d+k-3}}\geq 0.$$
By the previous computations, this is the same as:
$$3\cdot\bigg(27-12k-12d+4f_{1}-3f_{0}-t_{2}\bigg) -81 + 56k + 32d - 16f_{1} + 25f_{0}- t_{2}\geq 0,$$
which finally gives us
\[
5k+t_{2}+t_{3}\geq d + \sum_{r\geq 5}(r-4)t_{r},
\]
as claimed.
\end{proof}

The next result improves  \cite[Theorem 4.2]{PSz2023}:
\begin{proposition}\label{prop: miyaoka}
Let $\mathcal{CL} = \{\ell_{1}, \dots, \ell_{d}, C_{1}, \dots, C_{k}\} \subset \mathbb{P}^{2}_{\mathbb{C}}$ be an arrangement of $d\ge 3$ (or $d=0$) lines  and $k\ge 3$ smooth conics with   ordinary singularities.  Assume that:
\begin{itemize}
\item[(a)] there is no point $P\in \mathbb P^2$ that lies on all the lines and on no conic of $\mathcal{CL}$;
\item[(b)]  $\mathcal{CL}$ does not contain an exceptional line.
\end{itemize}
Then:
$$8k + t_{2} + \frac{3}{4}t_{3} \geq d + \sum_{r\geq 5}(2r-9)t_{r}.$$
\end{proposition}
\begin{proof}
The proof is similar to that of Proposition \ref{HirCL}, but we take $p=2$ and apply a log version of the Bogomolov--Miyaoka--Yau inequality due to Miyaoka.

Consider the cover $\sigma_2\colon W_2\to Y$. By Corollary \ref{cor: positivity} the canonical class $K_{W_2}=\sigma_2^*K_2$ is nef.
Therefore if   $A_{1}, \ldots , A_{m}$, $B_1,\ldots, B_n$ are smooth disjoint  curves in $W_2$ such that the $A_i$ are elliptic and the $B_j$ are rational with $B_j^2=-2$, then the following inequality holds by  \cite[Cor.~1.3]{M84}:
\begin{equation}\label{eq: miyaoka}
3e(W) - K^{2}_W \geq \sum_{i=1}^{k}(-A_{i}^{2}) +\frac{9}{2} n.
\end{equation} 
In our situation, let $E$ be an  exceptional curve of $\tau$ mapping to a   point of $\mathcal{CL}$ of multiplicity $r$ of $\CL$: if $r=4$ then  $\sigma_2^{-1}(E)$ is the disjoint union of $2^{k+d-5}$ ($2^{k-4}$ if $d=0$) elliptic curves of self--intersection $-4$, while  if $r=3$ then $\sigma_2^{-1}(E)$ is the disjoint union of $2^{k+d-4}$ ($2^{k-3}$ if $d=0$) rational curves of self--intersection $-2$.
 So \eqref{eq: miyaoka} gives
$$3e(W) - K_{W}^2 \geq 4t_{4}\cdot 2^{k+d-5}+\frac{9}{2}t_3 \cdot 2^{k+d- 4}$$
if $d\ge 3$ and 
$$3e(W) - K_{W}^2 \geq 4t_{4}\cdot 2^{k-4}+\frac{9}{2}t_3 \cdot 2^{k-3}$$
if $d=0$. 
After simple manipulations  in either case we arrive at 
$$8k + t_{2} + \frac{3}{4}t_{3} \geq d + \sum_{r\geq 5}(2r-9)t_{r},$$
as claimed.
\end{proof}
\begin{remark}\label{rem: p=3}
Assume that $\mathcal{CL}$ contains no exceptional line and that $k\ge 3$. 
Then arguing as in the proof of Proposition  \ref{prop: miyaoka}, but for $p=3$, we can derive the following inequality, which slightly improves Proposition \ref{HirCL}:
$$5k + t_{2} +\frac 34 t_{3} \geq d + \sum_{r\geq 5}(r-4)t_{r}.$$
\end{remark}

The next result is proven in \cite[Proposition~6.2]{PSz2023} by combinatorial methods under the additional assumption that $t_{d+k-1}=0$.
\begin{proposition}\label{prop: c2}
Let $\mathcal{CL} = \{\ell_{1}, \dots, \ell_{d}, C_{1}, \dots, C_{k}\} \subset \mathbb{P}^{2}_{\mathbb{C}}$ be an arrangement of $d$ lines and $k\geq 3$ smooth conics admitting only ordinary singularities and such that $t_{d+k}=0$. 
Then
  $$3-2k-2d+f_{1}-f_{0}>0.$$
\end{proposition}
\begin{proof} By Proposition  \ref{cor: positivity} we have $\bar c_1(Y,\CL)=K^2>0$. Since $K=\lim_pK_p$ we have $K_p^2>0$ for large $p$. Applying Bogomolov--Miyaoka--Yau inequality to $W_p$ for large $p$ we obtain:
\begin{equation}\label{eq: BMY}
\frac{3e(W_p)-K_{W_p}^2}{p^{d+k-1}}=\frac{3e(W_p)}{p^{d+k-1}}-K_p^2\ge 0.
\end{equation}
It is easy to check that $$\lim_p \frac{e(W_p)}{p^{d+k-1}}=e(Y\setminus \oCL)=\bar c_2(Y,\oCL),$$ hence  
taking the limit in \eqref{eq: BMY} and recalling \eqref{eq: chern-numbers} we obtain:
\begin{equation*}
3-2k-2d+f_{1}-f_{0} = \bar c_2(Y,\oCL)\ge\frac{K^2}{3}>0. \qedhere{}
\end{equation*}
\qedhere\end{proof}

\section{Bounds for the characteristic numbers of conic-line arrangements}\label{sec:bounds}
Now we are ready to prove the main result of our note.

\begin{theorem}
\label{charn}
Let $\mathcal{CL} \subset \mathbb{P}^{2}_{\mathbb{C}}$ be an arrangement of $d$ lines and $k\geq 3$ smooth conics such that $t_{k+d}=0$. Then $\gamma(\mathcal{CL}) < \frac{8}{3}$.
\end{theorem}
\begin{proof}
 Assume by contradiction that $\gamma(\mathcal{CL}) \geq \frac{8}{3}$, namely 
$$\frac{9-8k-5d+3f_{1}-4f_{0}}{3-2d-2k+f_{1}-f_{0}}\geq \frac{8}{3}.$$
By  Proposition \ref{prop: c2} (cf.  also \cite[Proposition 6.2]{PSz2023}), the denominator $3-2d-2k+f_{1}-f_{0}$ is strictly positive, so we get 
$$3+d+f_{1}-4f_{0}\geq 8k.$$
By the Hirzebruch--type inequality of Proposition \ref{HirCL}, we have 
$$5k-t_{2}\geq d+f_{1}-4f_{0},$$
which implies that
$$3+5k-t_{2} \geq 3+d+f_{1}-4f_{0} \geq 8k.$$
This gives us
$$3\geq 3k+t_{2},$$
but $k\geq 3$ and $t_{2}\geq 0$, so we arrive at a contradiction.
\end{proof}
\begin{corollary}
In the setting of Theorem \ref{HirCL}, we have
$$8k + 2t_{2}+t_{3} > 3+d + \sum_{r\geq 5}(r-4)t_{r}.$$
\end{corollary}
Now we provide a sharper bound on the characteristic numbers under the assumption that intersection points are only double and triple points.
\begin{proposition}
Let $\mathcal{CL} \subset \mathbb{P}^{2}_{\mathbb{C}}$ be an arrangement of $d$ lines and $k\geq 3$ smooth conics such that the only intersection points are ordinary double and triple points. Then
 $$\gamma(\mathcal{CL}) < \frac{5}{2}.$$
\end{proposition}
\begin{proof}
Recall that
$$\bar{c}_{2}(Y,\overline{\mathcal{CL}})=3-2d-2k+t_{2} + 2t_{3}> 0$$ 
by Proposition \ref{prop: c2}. In order to prove our statement, assume by contradiction $\gamma(\mathcal{CL}) \geq \frac{5}{2}$. It means that
$$\frac{9-8k-5d+2t_{2}+5t_{3}}{3-2d-2k+t_{2}+2t_{3}}\geq \frac{5}{2}.$$
This gives us
$$18 - 16k - 10d +4t_{2} + 10t_{3} \geq 15 - 10d - 10k + 5t_{2} + 10t_{3},$$
so finally we arrive at
$$3 \geq 6k+t_{2}.$$
Since $k\geq 3$, we get a contradiction.
\end{proof}

Finally  we  to present a sharp lower bound on the characteristic numbers for arrangements consisting of conics and only having ordinary singularities. Our result is in the spirit of \cite[Proposition 3.4]{geo}.

\begin{theorem}
Let $\mathcal{C}$ be an arrangement of $k\geq 3$ smooth conics in the plane such that $t_{k}=0$.
Then 
$$\gamma(\mathcal{C}) :=\frac{9-8k+3f_{1}-4f_{0}}{3-2k+f_{1}-f_{0}} \geq \frac{4k^2-12k+9}{2k^{2}-4k+3}.$$
Moreover, the lower bound is achieved when our arrangement has only double points as intersections.
\end{theorem}
\begin{proof}
Observe that if $\mathcal{C}$ admits only double intersection points, then
$$\gamma(\mathcal{C})= \frac{9-8k+2t_{2}}{3-2k+t_{2}} = \frac{4k^2-12k+9}{2k^{2}-4k+3}.$$
Now we want to show that the case when we have only double intersection points as singularities gives  a lower bound  for $\gamma(\mathcal{C})$. We wish  to show that
\begin{equation}
\frac{9-8k+3f_{1}-4f_{0}}{3-2k+f_{1}-f_{0}} \geq \frac{4k^2-12k+9}{2k^{2}-4k+3}.
\end{equation}
Observe that the above inequality is equivalent to
\begin{equation}
\label{main}
-8k^3+14k^2-6k + 2k^2f_{1}+(-4k^2+4k-3)f_{0} \geq 0.     
\end{equation}
Recall that the following combinatorial count holds
$$4\cdot \binom{k}{2} = 2(k^2-k) = \sum_{r\geq 2} \binom{r}{2}t_{r}.$$
Multiplying by $2$ the above formula we get
\begin{equation}
4k^2-4k  = \sum_{r \geq 2}r^{2}t_{r} - \sum_{r\geq 2} rt_{r}= f_{2} - f_{1}.
\end{equation}
and observe that
$$-8k^3+14k^2-6k = -2k(f_{2}-f_{1}) + \frac{3}{2}(f_{2}-f_{1}) = \bigg(-2k+\frac{3}{2}\bigg)f_{2} + \bigg(2k-\frac{3}{2}\bigg)f_{1}.$$
Plugging this into \eqref{main}, we obtain
$$\bigg(-2k + \frac{3}{2}\bigg)f_{2}+ \bigg(2k^2 + 2k-\frac{3}{2}\bigg)f_{1} +(-4k^2 + 4k - 3)f_{0} =$$
$$\sum_{r\geq 2}\bigg( \bigg(-2k+\frac{3}{2}\bigg)r^{2} + \bigg(2k^{2} + 2k - \frac{3}{2}\bigg)r + (-4k^{2} + 4k-3)\bigg)t_{r} \geq 0.$$
In order to finish the proof, we have to show that 
\begin{equation}
\label{check}
\bigg( \bigg(-2k+\frac{3}{2}\bigg)r^{2} + \bigg(2k^{2} + 2k - \frac{3}{2}\bigg)r + (-4k^{2} + 4k-3)\bigg) \geq 0
\end{equation}
for $r \in \{2, \ldots , k-1\}$ and suitably taken values of $k$, i.e., for $r \in \mathbb{N}_{\geq 2}$ the inequality must  hold with $k\geq r + 1$, since $t_{k}=0$. If we plug in $k = r + h$ with $h\geq 1$, then the left-hand side of \eqref{check} has the following form:
\begin{equation}
(r-2)\cdot(4h(h+r-1)-r+3).
\end{equation}
Since $r\geq 2$ and $h\geq 1$ we see that $(r-2)\cdot(4h(h+r-1)-r+3) \geq 0$, which completes our proof.
\end{proof}
Let us recall the world--record conic--line arrangements having the highest known characteristic number.
\begin{example}[Klein's arrangement of conics and lines]
In \cite{PR2019}, the second author with J. Ro\'e described in detail an interesting conic--line arrangement $\mathcal{CL}$ consisting of $21$ lines and $21$ conics, i.e., these curves are polars to Klein's quartic curve at the $21$ nodes of the associated Steinerian curve, and it has $42$ double points, $252$ triple points, and $189$ quadruple points. It is also worth noting that the group ${\rm PSL}(2;\mathbb{F}_{7})$ plays an important role in the geometry of this conic--line arrangement, and refer to \cite[Section 3]{PR2019} for more details. We can compute that
$$\gamma(\mathcal{CL}) = \frac{9 - 8\cdot 21 - 5\cdot 21 + 2\cdot 42 + 5\cdot 252 + 8\cdot 189}{3-2\cdot 21 - 2\cdot 21 + 42 + 2\cdot 252 + 3\cdot 189} \approx 2.512,$$  and to the best of our knowledge, this is the highest known value.
\end{example}
Finishing our note, let us formulate the following difficult conjecture which is strictly motivated by the prediction regarding characteristic numbers for both line and conic arrangements in the complex plane.
\begin{conjecture}
In the setting of Theorem \ref{charn}, for a fixed $\varepsilon >0$ there exist only finitely many conic--line arrangements $\mathcal{CL} \subset \mathbb{P}^{2}_{\mathbb{C}}$ with only ordinary singularities with $\gamma(\mathcal{CL}) > \frac{5}{2} + \varepsilon$.
\end{conjecture}
\section*{Acknowledgments}
 \noindent Rita Pardini is a member of Gnsaga-INdAM and she was partially supported by PRIN 2017SSNZAW\_004 and 2022BTA242. 
\section*{Conflict of Interests}
I declare that there is no conflict of interest regarding the publication of this paper.
\section*{Data Availability Statement}
We do not analyse or generate any datasets, because this work proceeds within a theoretical and mathematical approach.

\vskip 0.5 cm
Rita Pardini,
Dipartimento di Matematica, 
Universit\`a degli studi di Pisa, 
Largo Pontecorvo 5, 56127 Pisa, Italy. \\
\nopagebreak
\textit{E-mail address:} 
\texttt{rita.pardini@unipi.it}

\bigskip

Piotr Pokora,
Department of Mathematics,
University of the National Education Commission Krakow,
Podchor\c a\.zych 2,
PL-30-084 Krak\'ow, Poland. \\
\nopagebreak
\textit{E-mail address:} \texttt{piotr.pokora@uken.krakow.pl}
\bigskip

\begin{thebibliography}{000}
\bibitem{AP} V.~Alexeev, and R.~Pardini,  On the existence of ramified abelian covers. \textit{Rend. Sem. Mat.  Univ. Politec. Torino} \textbf{71  (3--4)}: 307 -- 315 (2013). 
\bibitem{BHH87}
G. Barthel, F. Hirzebruch, and Th. H\"ofer, \textit{Geradenkonfigurationen und algebraische Fl\"achen. Aspects of mathematics}.  D4. Vieweg, Braunschweig, (1987).
\bibitem{geo}
S. Eterovi\'c, F. Figueroa, and G. Urz\'ua, On the geography of line arrangements. \textit{Adv. Geom.} \textbf{22(2)}: 269--276 (2022).
\bibitem{Hirzebruch83}
F. Hirzebruch, Arrangements of lines and algebraic surfaces. \textit{Arithmetic and geometry, Vol.II, Progr. Math., vol.} \textbf{36}, Birkh\"auser Boston, Mass.: 113--140 (1983).
\bibitem{Hirzebruch}
F. Hirzebruch, Singularities of algebraic surfaces and characteristic numbers. \textit{The Lefschetz centennial conference, Part I (Mexico City, 1984) Contemp. Math.} \textbf{58}: 141 -- 155 (1986).
\bibitem{Iitaka}
S. Iitaka, Geometry on complements of lines in $\mathbb{P}^{2}$. \textit{Tokyo J. Math.} \textbf{1}: 1 -- 19 (1978).
\bibitem{Miyaoka77}
Y. Miyaoka, On the Chern numbers of surfaces of general type. \textit{Invent. Math.} \textbf{42(1)}: 225 -- 237 (1977).
\bibitem{M84}
Y. Miyaoka, The maximal number of quotient singularities on surfaces with given numerical invariants. \textit{Math. Ann.} \textbf{268(2)}: 159 -- 171 (1984).
\bibitem{RP}
R. Pardini, Abelian covers of algebraic varieties. \textit{J. Reine Angew. Math.} \textbf{417}: 191 -- 213 (1991).
\bibitem{Pokora1}
P. Pokora, Hirzebruch type inequalities and plane curve configurations. \textit{International Journal of Mathematics} \textbf{28(2)}: 1750013 (11 pages) (2017).
\bibitem{PR2019}
P. Pokora and J. Ro\'e, The 21 reducible polars of Klein's quartic. \textit{Exp. Math.} \textbf{30(1)}: 1 -- 18 (2021).
\bibitem{PSz2023}
P. Pokora and T. Szemberg, Conic--line arrangements in the complex projective plane. \textit{Discrete Comput. Geom} \textbf{69(4)}: 1121 -- 1138 (2023).
\bibitem{Som}
A. J. Sommese, \newblock On the density of ratios of Chern numbers of algebraic surfaces. \textit{Math. Ann.} \textbf{268(2)}: 207 -- 221 (1984).
\bibitem{Tang}
L. Tang, Algebraic surfaces associated to arrangements of conics. \textit{Soochow J. Math.} \textbf{21(4)}: 427 -- 440 (1995).
\bibitem{Urzua}
G. Urz\'ua, \emph{Arrangements of curves and algebraic surfaces}. Ph.D. Thesis, University of Michigan, 2008. 
\bibitem{Wahl}
J. Wahl, Miyaoka--Yau inequality for normal surfaces and local analogues. \newblock Ciliberto, Ciro (ed.) et al., \textit{Classification of algebraic varieties. Algebraic geometry conference on classification of algebraic varieties, May 22-30, 1992, University of L\textquoteright Aquila, L\textquoteright Aquila, Italy. Providence, RI: American Mathematical Society. Contemp. Math.} 162, 381-402 (1994). 
\end{thebibliography}
\end{document}